\documentclass[12pt]{article}

\usepackage{mymacros}
\usepackage[all]{xy}

\newcommand{\Coh}{\operatorname{Coh}}

\newcommand{\Auteq}{\operatorname{Auteq}}

\makeatletter
\newcommand{\subjclass}[2][2010]{%
  \let\@oldtitle\@title%
  \gdef\@title{\@oldtitle\footnotetext{#1 \emph{Mathematics subject classification.} #2}}%
}
\newcommand{\keywords}[1]{%
  \let\@@oldtitle\@title%
  \gdef\@title{\@@oldtitle\footnotetext{\emph{Key words and phrases.} #1.}}%
}
\makeatother

\title{Exceptional sheaves \\ on the Hirzebruch surface $\bF_2$}
\author{Shinnosuke Okawa and Hokuto Uehara}
\date{}

\begin{document}

\maketitle

\begin{abstract}
We investigate exceptional sheaves on the Hirzebruch surface
$
\bF_2
$,
as the first attempt toward the classification of
exceptional objects on weak del Pezzo surfaces.
\end{abstract}

\tableofcontents

\section{Introduction}
Let
$
X
$
be a smooth projective variety over an algebraically closed field
$
k
$.
The bounded derived category of coherent sheaves
on
$
X
$,
which we denote by
$
D(X)
=
D^b \coh X
$,
admits a structure of
$
k
$-linear triangulated category (see \cite{MR2244106}).
An object
$
\cE
$
of a
$
k
$-linear triangulated category is called \emph{exceptional}
if it satisfies the condition
\[
\Hom(\cE , \cE[i]) =
\begin{cases}
0 & i\ne 0 \\
k\cdot id_\cE & i=0
\end{cases}
\]
(see \cite[Definition 1.57]{MR2244106}).
For example,
when
$
X
$
is a Fano manifold in characteristic zero,
the Kodaira vanishing theorem implies that
line bundles on
$X$
are exceptional objects of
$
D(X)
$.
Classification of exceptional objects in a given triangulated category
is basic but quite a nontrivial issue.

Exceptional objects of the derived category of del Pezzo surfaces (i.e., Fano manifolds of dimension $2$)
were thoroughly studied in the paper \cite{MR1286839}.
Among others, it was shown in \cite[Propositions 2.9 and 2.10]{MR1286839} that any exceptional object
on such surfaces is isomorphic to a shift of an exceptional vector bundle or a line bundle
on a $(-1)$-curve.

Since exceptional objects on del Pezzo surfaces are well understood,
it is natural to work on \emph{weak del Pezzo} surfaces;
i.e., those non-singular surfaces with nef and big anti-canonical line bundles.
The purpose of this paper is to study
exceptional sheaves on the Hirzebruch surface of degree $2$, i.e.,
$
\bF_2
=
\bP_{\bP^1}{(\cO_{\bP^1} \oplus \cO_{\bP^1} (-2))}
$,
which is the easiest example of weak del Pezzo surfaces.

Unlike the case of del Pezzo surfaces the situation becomes much more involved.
This is due to the existence of the non-standard autoequivalences
of the derived category which are called \emph{twist functors}.

\begin{definition}[\cite{Seidel-Thomas}]\label{df:twist}
Let
$
X
$ be a smooth projective variety. 
\begin{enumerate}[(1)]
\item
We say that an object $\alpha \in D(X)$ is \emph{spherical} if 
we have $\alpha \otimes \omega_{X} \cong \alpha$ and
$$
\Hom (\alpha,\alpha[i])
\cong
\begin{cases}  0 & i\ne 0,\dim X\\
                                                k\cdot id_\alpha & i=0,\dim X. 
\end{cases}
$$

\item
Let $\alpha\in D(X)$ be a spherical object. 
We consider the mapping cone 
$$
\mathcal{C}=\Cone(\pi_1^*\alpha^\vee\otimes \pi_2^*\alpha \to \mathcal{O}_{\Delta})
$$
of the natural evaluation
$
\pi_1^*\alpha^\vee\otimes \pi_2^*\alpha \to \mathcal{O}_{\Delta}
$,
where
$
\Delta \subset X \times X
$
is the diagonal and $\pi_i$ is the projection from
$X\times X$ to the $i$th factor. Then the integral functor 
$T_{\alpha}:=\Phi^{\mathcal{C}}_{X\to X}$ defines an autoequivalence of $D(X)$, 
called the \emph{twist functor} along the spherical object $\alpha$. 
By definition, for $\beta\in D(X)$, we have an exact triangle
\begin{equation}\label{eq:twist_def}
\RHom(\alpha,\beta)\otimes
\alpha\xrightarrow[]{\textrm{evaluation}}
\beta \longrightarrow
T_{\alpha}\beta.
\end{equation}
We can also define the \emph{inverse twist functor} $T'_{\alpha}$
so that it is a quasi-inverse of $T_{\alpha}$ 
and there exists an exact triangle
\begin{equation}\label{eq:inverse-twist_def}
T'_{\alpha}\beta \longrightarrow \beta\xrightarrow[]{\textrm{coevaluation}}\RHom(\beta, \alpha)^{\vee}\otimes
\alpha.
\end{equation}
\end{enumerate}
\end{definition}

\begin{remark}\label{rm:spherical_twist_on_K_0}
The triangle \pref{eq:twist_def} yields the equality 
$$
[ T_{\alpha} \beta ]=[\beta]-\chi (\alpha,\beta)[\alpha]
$$
in the
$K$-group
$
K_0(X)
$. Here $\chi(-,-)$ denotes the Euler pairing.
From this and \pref{df:twist} (1), 
we see that
$$
\chi (\alpha,T_\alpha\beta)=(-1)^{\dim X+1}\chi (\alpha,\beta).
$$
As a consequence, if $\dim X$ is even,
the
(inverse) twist functor acts as a reflection
on 
$
K_0(X)
$.
In particular, if we do it twice, the action on
$K_0(X)$
is trivial.
\end{remark}

Note that any line bundle on the ($-2$)-curve on $\bF_2$
provides an example of a spherical object in $D(\bF_2)$.
Twisting exceptional sheaves by those spherical objects,
we can produce many more exceptional objects on $\bF_2$ and
they are not necessarily isomorphic to shifts of sheaves.
Note that the group of autoequivalences of $\bF_2$ is known by
\cite[Theorem 1]{MR3162236} and satisfies
\[
\Auteq{(D(\bF_2))}
=
(B(\bF_2) \rtimes \bZ\cO(F) )
\times \Aut(\bF_2)\times \bZ[1],
\]
where
$
F
$
is a fiber of the
$
\bP^1
$-bundle
$
\bF_2 \to \bP^1
$
and
$B(\bF_2)$ is the subgroup generated by spherical twists.
Here
$ [ 1 ] $ is the shift by $ 1 $ of complexes,
and
$
 \cO ( F )
$
acts by tensoring with this line bundle.
Since
$
 \Aut ( \bF_2 ) = \Aut^0 ( \bF_2 )
$,
and line bundles and spherical objects are rigid,
we see that elements of
$
 \Aut ( \bF_2 )
$
commute with all other autoequivalences.

Despite this complication, we expect the following conjecture.

\begin{conjecture}\label{cj:main_conjecture}
For any exceptional object $\cE\in D(\bF_2)$, there exists
an autoequivalence $\Phi\in\Auteq{(D(\bF_2))}$ such that
$\Phi(\cE)$ is an exceptional vector bundle on $\bF_2$.
\end{conjecture}

As a special case, we prove the following theorem.
We denote by
$
C \subset \bF_2
$
the unique
$
(-2)
$-curve.


\begin{theorem}\label{th:main_conjecture;the_case_of_sheaves}
Let $\cE$ be an exceptional sheaf on $\bF_2$ with a nontrivial torsion subsheaf.
Consider the standard decomposition of the sheaf $\cE$ 
\begin{equation}\label{eq:TEF_sequence}
0\to \cT^-\to \cE\to \cF^-\to 0
\end{equation}
into the torsion part $\cT^-$ and the torsion-free part $\cF^-$.
Also let 
\begin{equation}\label{eq:FET_sequence}
0\to \cF^+\to \cE\to \cT^+\to 0
\end{equation}
be the standard decomposition of the sheaf $\cE$ into
the restriction $\cT^+$ of $\cE$ to $C$ 
and the kernel $\cF^+$ of the restriction map.
Then the following hold.

	\begin{enumerate}[(1)]
	\item\label{it:Excis_the_inverse_twist_of_F}
	There exists an integer $a$ satisfying the following properties.
		\begin{itemize}
		\item
		The triangle \pref{eq:twist_def} for $\alpha=\cO_C(a)$ and $\beta=\cE$
		becomes a short exact sequence
		\begin{equation}\label{eq:the_inverse_twist_sequence}
		0
		\to
		\Hom{}( \cO_C(a),\cE)\otimes \cO _C(a)
		\to
		\cE
		\to
		T_{\cO_C(a)} \cE
		\to
		0,
		\end{equation}
		and is isomorphic to the sequence \pref{eq:TEF_sequence}.
		
		\item
		The triangle
		\pref{eq:inverse-twist_def} for $\alpha=\cO_C(a+1)$ and $\beta=\cE$
		becomes a short exact sequence
	\begin{equation}\label{eq:the_twist_sequence}
	0
	\to T'_{\cO_C(a+1)}\cE
	\to  \cE
	\to  \Hom(\cE, \cO_C(a+1))^\vee\otimes \cO _C(a+1)
	\to 0,
	\end{equation}
		and is isomorphic to
		the sequence \pref{eq:FET_sequence}.
		\end{itemize}

	 \item\label{it:F_is_an_exceptional_vector_bundle}
	$\cF^+$ and $\cF^-$ are exceptional vector bundles. Moreover they are related to each other by 
	$$\cF^-\cong \cF^+\otimes \cO (C).$$	
	\end{enumerate}
\end{theorem}

\begin{remark}
\pref{cj:main_conjecture}
is not correct for Hirzebruch surfaces
$\bF_{n}$
with
$n\ge 3$.
In fact, since those surfaces have no non-standard autoequivalences by
\cite[Theorem 1]{MR3162236},
shifts of sheaves are always sent to shifts of sheaves under autoequivalences.
On the other hand, we can construct exceptional objects which are
genuine complexes as follows.
Let
$F, C$ be the classes of a fiber and the negative curve, respectively.
Then we easily see that
\begin{equation*}
\cO , \scL:= \cO((n-1)F+C)
\end{equation*}
is an exceptional pair.
The left mutation (see \cite{Bondal-Kapranov_Serre}) of the pair is included in the triangle
\begin{equation*}
L_{\cO}\scL \to \RHom(\cO, \scL)\otimes \cO \xrightarrow[]{\mathrm{evaluation}} \scL,
\end{equation*}
and we can easily check that
$\scH^{0}$ and
$\scH^{1}$
of the complex
$L_{\cO}\scL$
are nontrivial; to see this, note that
$\scL$
is not globally generated since
$\scL \cdot \cO(C)=-1<0$.

\end{remark}

For a del Pezzo surface
$
X
$,
any exceptional object of
$D(X)$
is, up to even shifts
$\bZ[2]$,
uniquely determined by its class in
$K_0(X)$ (see \cite[Corollary 2.5]{MR966982}).
This is not the case for
$\bF_2$ (see \pref{rm:action_of_double_spherical_twists}).
Therefore we can also pose the following finer problem.

\begin{problem}
For each exceptional object
$
\scEtilde
\in
D(\bF_2)
$,
classify all other exceptional objects
$\cE$
such that
$
[\scEtilde]=[\cE]
\in
K_0(\bF_2)
$.
\end{problem}

As a partial answer to the problem,
for any exceptional object
$
\scEtilde
$
on
$\bF_2$,
we describe all the exceptional sheaves
$\cE$
sharing the class with
$
\scEtilde
$.

\begin{theorem}[{\pref{cr:classification_of_equivalent_sheaves} $+$ \pref{pr:existence_of_sheaf_in_the_class}}]
For any exceptional object
$
\scEtilde
\in
D(\bF_2)
$,
there exists a unique exceptional vector bundle
$\cE$
such that
$
[\scEtilde]
=
[\cE]
$
or
$
-[\cE]
\in
K_0(\bF_2)
$.
Moreover the set of exceptional sheaves sharing the class with
$
\cE
$
can be explicitly described as
$
\{
\cE_i \ | \ 
i \ge -1
\}
$
with
$\cE_{-1}=\cE$
(see
\pref{sc:Exceptional_sheaves_sharing_classes}
for the definition of
$\cE_i$).

\end{theorem}

\subsection*{Acknowledgements}
The authors would like to thank Daniel Huybrechts for answering
questions about stability of sheaves.
They are also indebted to the referee for a very careful
reading of the paper and for his/her
helpful comments and corrections.
A part of this work was done during the stay of the authors
at the Max-Planck-Institut f\"ur Mathematik.
They are grateful for their hospitality and support.
S.~O. was partly supported by JSPS Grant-in-Aid for Young Scientists No.~25800017.
H.~U. was partly supported by Grant-in-Aid for Scientific Research No.~23340011.
\subsection*{Notations and Conventions}
We always work over an algebraically closed field
$
k
$
of an arbitrary characteristic.

To simplify the notations, we write
$
{}^i(\cE,\cF)
$
to indicate
$\Ext^i(\cE,\cF)=\Hom(\cE,\cF[i])$.
Even more,
${}^0(\cE, \cF)$
is also denoted by
$(\cE, \cF)$.
These symbols also indicate the dimensions of the respective vector spaces,
depending on the context.
The symbol
$
\chi ( - , -)
$
denotes the Euler pairing, which is defined by 
$
\chi(\cE,\cF)=\sum_{i\in \mathbb{Z}} (-1)^i\dim {}^i(\cE,\cF).
$

For each
$
n>0
$,
we denote by
$
C
\subset
\bF_n
$
the unique negative curve with
$
C^2
=
-n
$.
Exceptional objects on surfaces whose anti-canonical line bundle is big and
has at most zero-dimensional base locus were systematically
investigated in \cite{MR1604186}.
Since the anti-canonical line bundle of
$
\bF_2
$
is globally generated, we freely quote the results of \cite{MR1604186}
in this paper. In particular we use the notion of
$(-K_{\bF_2})$-stability for torsion-free sheaves on this surface.

\section{Proof of \pref{th:main_conjecture;the_case_of_sheaves}}

We give a proof of \pref{th:main_conjecture;the_case_of_sheaves} in this section.
For simplicity, we put $\cF=\cF^-$ and $\cT=\cT^-$.
Let us begin with some preparations.

\begin{lemma}\label{lm:dimensions_of_ext_groups_in_TEF_sequence}
For any sequence of the form \pref{eq:TEF_sequence} with $\cE$ exceptional,
the dimensions of the $\Ext$ groups among $\cT,\cE,\cF$ can be calculated as follows:

\begin{alignat*}{3}
{}^0(\cT, \cT)  &=  t &\quad\quad\quad {}^1(\cT, \cT) &=  0&\quad\quad {}^2(\cT, \cT) & =  t\\
{}^0(\cT, \cE)  &=  t & {}^1(\cT, \cE) &=  f-1& {}^2(\cT, \cE) &=  0\\
{}^0(\cT, \cF)  &= 0 & {}^1(\cT, \cF) &=  f+t-1& {}^2(\cT, \cF) &=  0\\
{}^0(\cE, \cT) &=  0& {}^1(\cE, \cT) &=  f-1& {}^2(\cE, \cT) &=  t\\
{}^0(\cE, \cE)  &=  1& {}^1(\cE, \cE) &=  0& {}^2(\cE, \cE) &=  0\\
{}^0(\cE, \cF)  &=  f& {}^1(\cE, \cF) &=  t& {}^2(\cE, \cF) &=  0\\
{}^0(\cF, \cT)  &=  0& {}^1(\cF, \cT) &=  f+t-1& {}^2(\cF, \cT) &=  0\\
{}^0(\cF, \cE)  &=  0& {}^1(\cF, \cE) &=  t-1& {}^2(\cF, \cE) &=  f-1\\
{}^0(\cF, \cF)  &=  f& {}^1(\cF, \cF) &=  0& {}^2(\cF, \cF) &=  f-1
\end{alignat*}
Above we set $t={}^0(\cT, \cT)$ and $f={}^0(\cF, \cF)$.
\end{lemma}

\begin{proof}
Consider the long exact sequences associated to the six functors
\[
\begin{split}
(\cE,-), (-, \cE), (\cF, -), (-, \cF), (\cT, -), (-, \cT)
\end{split}
\]
applied to the short exact sequence \pref{eq:TEF_sequence}, and
use the facts that
	\begin{itemize}
	\item $\cE$ is exceptional,
	\item $(\cT, \cF)=0$,
	\item $\cT\otimes \omega_{\bF_2}\cong \cT$ ($=$\cite[Lemma 2.2.5]{MR1604186}), and
	\item $\cF$ and $\cT$ are rigid (this follows from
	Mukai's lemma \cite[Lemma 2.1.4. 2. (a)]{MR1604186}
	and the rigidity of $\cE$).
	\end{itemize}
	\end{proof}

\begin{lemma}\label{lm:TEF_Lemma}
For any sequence of the form \pref{eq:TEF_sequence},
the sheaf $\cF$ is an exceptional vector bundle
and $\cT$ is  a pure one-dimensional sheaf
supported on $C$. 
\end{lemma}

\begin{proof}
First of all, \cite[Corollary 2.2.3]{MR1604186} implies that the torsion-free part $\cF$
and the torsion part $\cT$ are both rigid, and that $\cT$ is a pure one-dimensional sheaf.
Then \cite[Lemma 2.2.1]{MR1604186} tells us that $\cF$ is a vector bundle.

Next, by \cite[Theorem 2.4.1. (1)]{MR1604186} we see that
$\cF$ is a direct sum
\[\cF=\bigoplus_{i=1}^{m}\cF_i\]
of $\mu_{(-K_{\bF_2})}$-semi-stable rigid bundles.
If we assume $m>1$, from \cite[Lemma 2.2.2]{MR1604186}
we obtain the inequality $(\cF,\cF)\ge {}^2(\cF,\cF)-2$.
Since \pref{lm:dimensions_of_ext_groups_in_TEF_sequence}
implies $(\cF,\cF)-{}^2(\cF,\cF)=1$, this is a contradiction.

Hence we see that
$m=1$, so that
$\cF$ itself is semi-stable.
This implies
$
{}^2(\cF,\cF)
=
(\cF, \cF\otimes\omega_{\bF_2})
=
0
$,
because of the inequality
\[
\mu_{(-K_{\bF_2})}(\cF)>\mu_{(-K_{\bF_2})}(\cF\otimes\omega_{\bF_2})
\]
and
\cite[Lemma 1.1.5]{MR1604186}.
Therefore we see ${}^2(\cF, \cF)=(\cF, \cF)-1=0$, concluding that
$
\cF
$
is an exceptional vector bundle.
\end{proof}

\begin{remark}
The above proof in particular tells us that the number $f=(\cF, \cF)$ in
\pref{lm:dimensions_of_ext_groups_in_TEF_sequence} is one.
\end{remark}

For a positive integer $m$, let $\iota\colon C\hookrightarrow mC$ be the
natural closed immersion into the
$m$-th thickening of
$C$.
We next consider the Harder-Narasimhan filtrations of pure sheaves on
$mC$.

\begin{lemma}\label{lm:structure_of_pure_sheaves_supported_on_mC}
Let $\cT$ be a pure one-dimensional sheaf on the scheme $mC$.
Then the subquotients of the Harder-Narasimhan filtration
\begin{equation*}\label{eq:Harder-Narasimhan_filtration}
0
=
\cT^0
\subset
\cT^1
\subset
\cdots
\subset
\cT^n
=
\cT
\end{equation*}
 of $\cT$ are of the form
\[
\lc
\cT^1 / \cT^0 ,
\cT^2 / \cT^1 ,
\dots ,
\cT^n / \cT^{n-1}
\rc
=
\lc
\cO_C(a_1)^{\oplus r_1}, \cO_C(a_2)^{\oplus r_2}, \dots, \cO_C(a_n)^{\oplus r_n}
\rc
\]
with $a_1>a_2>\cdots>a_n$ and $r_i>0$.
\end{lemma}

\begin{proof}
A stable sheaf $S$ on $mC$ is simple \cite[Corollary 1.2.8]{Huybrechts-Lehn},
and thus is isomorphic to a coherent $\cO_C$-module by \cite[Lemma 4.8]{Ishii-Uehara_ADC}.
Since $C$ is isomorphic to $\bP^1$ and $S$ is stable, $S$ has to be a line bundle on $C$.

Recall that any semi-stable sheaf has a Jordan-H\"older filtration
\cite[Section 1.5]{Huybrechts-Lehn}.
By definition, its subquotients are stable with the same slope.
On the other hand, for any
$
a
$
and
$
m > 0
$
we have
\[
\begin{array}{ll}
& \Ext^1_{mC} ( \iota_* \cO_C(a) , \iota_* \cO_C(a) ) \\
\cong &
\Hom_{mC} ( \iota_* \cO_C(a) , \iota_* \cO_C(a) \otimes \omega_{mC} )^{\vee}
\\
\cong & 
\Hom_C(\cO_C(a), \cO_C(a-2m))^{\vee}
=
0.
\end{array}
\]
Therefore any semi-stable sheaf on
$
mC
$ turns out to be polystable;
i.e., isomorphic to the direct sum of its Jordan-H\"older factors.
Thus we conclude the proof.
\end{proof}

Now let $\cT$ be the torsion sheaf in \pref{eq:TEF_sequence}.
Since it is pure by \pref{lm:TEF_Lemma}, we can apply \pref{lm:structure_of_pure_sheaves_supported_on_mC}.
Then we can prove the following
\begin{claim}
The number $n$ of the Harder--Narasimhan factors of $\cT$ is $1$, so that
$\cT\cong\cO_C(a)^{\oplus r}$ holds for some integers
$a$ and $r>0$.
\end{claim}

\begin{proof}
Suppose for a contradiction that $n>1$.
Consider the short exact sequence
\begin{equation*}
0\to \cT^n/\cT^{n-1}\to \cE/\cT^{n-1}\to \cE/\cT^n=\cF\to 0
\end{equation*}
and apply the functor $(\cO_C(a_{n-1}), -)$ to it.
As part of the associated long exact sequence we obtain
\[
 {}^1( \cO_C ( a_{n-1} ), \cF )
 \xrightarrow[]{\delta}
 {}^2( \cO_C ( a_{n-1} ), \cT^n / \cT^{n-1} )
 \to
 {}^2( \cO_C ( a_{n-1} ), \cE / \cT^{n-1} ).
\]
The third term is trivial. In fact we have the series of inequalities
\[
\begin{array}{llr}
& {}^2(\cO_C(a_{n-1}), \cE/\cT^{n-1})^{\vee} & \\
= &
(\cE/\cT^{n-1}, \cO_C(a_{n-1})) & \mathrm{(Serre \ duality)} \\
\le &
(\cE/\cT^{n-1}, \cO_C(a_1)) & (a_{n-1} < a_1 ) \\
\le &
(\cE/\cT^{n-1}, \cT) & ( \cO_C ( a_1 ) \hookrightarrow \cT ) \\
\le &
(\cE, \cT) & ( \cE \twoheadrightarrow \cE / \cT^{n-1} ) \\
= &
0 & ( \mathrm{ \pref{lm:dimensions_of_ext_groups_in_TEF_sequence} } ).
\end{array}
\]
Therefore
$ \delta $ is surjective and we obtain the inequalities
\begin{equation}\label{eq:lower_bound}
{}^1(\cO_C(a_{n-1}), \cF)\ge {}^2(\cO_C(a_{n-1}), \cT^n/\cT^{n-1})=r_n(a_{n-1}-a_n+1)
\ge 2r_n.
\end{equation}

On the other hand, we can prove the two inequalities
\begin{equation}\label{eq:1(OC(an-1),F)<1(OC(an),F)}
{}^1(\cO_C(a_{n-1}), \cF)< {}^1(\cO_C(a_n), \cF)
\end{equation}
and
\begin{equation}\label{eq:1(OC(an),F)>=rn}
{}^1(\cO_C(a_n), \cF) \le r_n,
\end{equation}
to obtain a contradiction.

Let us show first the inequality
\pref{eq:1(OC(an-1),F)<1(OC(an),F)}.
Due to the Serre dualities on
$\bF_2$
and
$C$,
\pref{eq:1(OC(an-1),F)<1(OC(an),F)}
can be rewritten as
\begin{equation}\label{eq:equivalent_inequality}
h^0(C, \cO_C(-2-a_{n-1}) \otimes \cF|_C)
<
h^0(C, \cO_C(-2-a_{n}) \otimes \cF|_C).
\end{equation}
Since
$
r_n > 0
$
by the assumption,
combining with
\pref{eq:lower_bound},
we see that the LHS of
\pref{eq:equivalent_inequality}
$=$ the LHS of
\pref{eq:1(OC(an-1),F)<1(OC(an),F)}
is positive.
In particular,
$
\cO_C(-2-a_{n-1}) \otimes \cF|_C
$
contains at least one line bundle with non-negative degree
as a direct summand.
Thus the strict inequality
$a_{n-1}>a_n$
implies the strict inequality
\pref{eq:equivalent_inequality}.

In the remainder we prove \pref{eq:1(OC(an),F)>=rn}.
Note first that
\pref{eq:1(OC(an),F)>=rn}
is equivalent to
\begin{equation}\label{eq:proof;n=1;key_inequality}
{}^1(\cT/\cT^{n-1}, \cF)\le {}^2(\cT/\cT^{n-1}, \cT/\cT^{n-1}).
\end{equation}
In fact, note that the LHS of \pref{eq:proof;n=1;key_inequality}
equals
$r_n\cdot {}^1(\cO_C(a_n), \cF)$
and the RHS can be calculated as
\[
\begin{array}{ll}
& {}^2(\cT/\cT^{n-1}, \cT/\cT^{n-1})
=
(\cT/\cT^{n-1}, \cT/\cT^{n-1})
\quad (\textrm{Serre \ duality})
\\
= &
(\cO_C(a_n)^{\oplus r_n} , \cO_C(a_n)^{\oplus r_n})
= r_n^2.
\end{array}
\]
In order to show \pref{eq:proof;n=1;key_inequality}, consider the following short exact
sequence
\begin{equation}\label{eq:n-1-th-TEF_sequence}
0\to \cT/\cT^{n-1}\to \cE/\cT^{n-1}\to \cE/\cT=\cF\to 0.
\end{equation}
By applying the functor $(\cT/\cT^{n-1}, -)$ we obtain the exact sequence
\[{}^1(\cT/\cT^{n-1}, \cE/\cT^{n-1})\to {}^1(\cT/\cT^{n-1}, \cF)\to {}^2(\cT/\cT^{n-1}, \cT/\cT^{n-1}).\]
Hence it is enough to show that ${}^1(\cT/\cT^{n-1}, \cE/\cT^{n-1})=0$,
which is equivalent to
\begin{equation}\label{eq:the_key_vanishing}
{}^1(\cE/\cT^{n-1}, \cT/\cT^{n-1})=0
\end{equation}
by the Serre duality.
To show the latter, we apply the functor $(\cE/\cT^{n-1}, -)$ to the sequence \pref{eq:n-1-th-TEF_sequence}
to obtain the exact sequence
\[
\begin{array}{ll}
& (\cE/\cT^{n-1}, \cE/\cT^{n-1})
\xrightarrow[]{\epsilon}
(\cE/\cT^{n-1}, \cF)
\to
 {}^1(\cE/\cT^{n-1}, \cT/\cT^{n-1})
 \\
 \to
 &
 {}^1(\cE/\cT^{n-1}, \cE/\cT^{n-1})
 =
 0.
\end{array}
\]
The vanishing of the last entry, namely the rigidity of $\cE/\cT^{n-1}$,
can be checked by applying Mukai's lemma 
\cite[Lemma 2.1.4. 2.(a)]{MR1604186}
to the short exact sequence
\[
0
\to
\cT^{n-1}
\to
\cE
\to
\cE / \cT^{n-1}
\to
0
\]
and using the rigidity of $\cE$. For this
we have to check the vanishing of $(\cT^{n-1}, \cE/\cT^{n-1})$, and this follows from the fact that
$\cT^{n-1}$ is an extension of the line bundles $\cO_C(d)$ with $d>a_n$ and that
$\cE/\cT^{n-1}$ is an extension of line bundles $\cO_C(a_n)$ and the torsion free sheaf $\cF$.

Now since the map $\epsilon$ above is nontrivial, it is enough to show
$(\cE/\cT^{n-1}, \cF)=1$ for the vanishing
\pref{eq:the_key_vanishing}.
For this, we apply the functor $(- ,\cF)$ to \pref{eq:n-1-th-TEF_sequence}
to obtain the exact sequence
\[0\to (\cF, \cF)\to (\cE/\cT^{n-1}, \cF)\to (\cT/\cT^{n-1}, \cF).\]
Since $\cT/\cT^{n-1}$ is torsion and $\cF$ is torsion free, we see
$(\cT/\cT^{n-1}, \cF)=0$. Combining this with the fact $\cF$ is exceptional
(see \pref{lm:TEF_Lemma}),
we obtain
$(\cE/\cT^{n-1}, \cF)=1$ and thus conclude the proof.
\end{proof}

Let us now conclude the proof of
\pref{th:main_conjecture;the_case_of_sheaves}.
So far we have seen that the sequence \pref{eq:TEF_sequence}
is always of the form
\begin{equation}\label{eq:TEF_sequence;final_form}
0
\to
\cO_C(a)^{\oplus r}
\to
\cE
\to
\cF
\to
0
\end{equation}
for some $a\in \bZ$ and $r\in\bZ_{>0}$, with
$
\cF
$
an exceptional vector bundle.
Moreover by 
${}^1( \cO_C(a),\cE)={}^2( \cO_C(a),\cE)=0$ in  \pref{lm:dimensions_of_ext_groups_in_TEF_sequence},
we see that the exact triangle
\pref{eq:twist_def} for $\alpha=\cO_C(a)$ and $\beta=\cE$
gives rise to a short exact sequence
\pref{eq:the_inverse_twist_sequence}.

Let us check that the sequence \pref{eq:TEF_sequence;final_form}
is isomorphic to \pref{eq:the_inverse_twist_sequence}.
For this note first that
$$
(\cO_C(a),\cE) =\frac{t}{r}=r
$$
follows from
\pref{lm:dimensions_of_ext_groups_in_TEF_sequence}
and $f=1$.
On the other hand since
$
( \cE , \cF )
\cong
( \cF , \cF )
=
k
$
by \pref{lm:dimensions_of_ext_groups_in_TEF_sequence},
the two morphisms
$
\cE \twoheadrightarrow \cF
$
in
\pref{eq:TEF_sequence;final_form}
and
\pref{eq:the_inverse_twist_sequence}
should be isomorphic. Thus we obtain
an isomorphism between these two short exact sequences.
In particular we obtain
$\cE\cong T'_{\cO_C(a)} \cF$.
Thus we obtain the first half of
\pref{th:main_conjecture;the_case_of_sheaves}
\pref{it:Excis_the_inverse_twist_of_F}.

Consider the sequence 
\pref{eq:TEF_sequence;final_form}
and its restriction to $C$. 
Then we obtain the following commutative diagram:
\begin{align*}
\xymatrix{
0 \ar[r] & \cO_C(a)^{\oplus r}\ar[r]\ar@{=}[d] & \cE \ar[r]\ar@{->>}[d] & \cF\ar[r]\ar@{->>}[d] & 0
\\
0\ar[r]  & \cO_C(a)^{\oplus r}\ar[r]       & \cE|_C \ar[r]          & 
\cF|_C \ar[r]   & 0.}
\end{align*}
We see from the snake lemma
that the sequence  \pref{eq:FET_sequence} is of the form  
\begin{equation}\label{eq:restriction_of_E}
0\to \cF(-C) \to \cE\to \cE|_C \to 0.
\end{equation}
This means that $\cF(-C)\cong \cF^+$, which is
\pref{th:main_conjecture;the_case_of_sheaves}.
\pref{it:F_is_an_exceptional_vector_bundle}.
On the other hand, by
\cite[Lemma 4.15]{Ishii-Uehara_ADC} we have
 $$
 T'_{\cO_C(a+1)} \cE
 \cong
 T'_{\cO_C(a+1)} \circ T'_{\cO_C(a)} \cF
 \cong
 \cF(-C).
 $$   
Therefore the exact triangle
\pref{eq:inverse-twist_def} for
$
\alpha = \cO_C(a+1)
$
and
$
\beta=\cE
$
gives rise to a short exact sequence
\begin{equation}\label{eq:F(-C)_exact_sequence}
0\to \cF(-C) \to \cE \to (\cE,\cO_C(a+1))^\vee\otimes \cO_C(a+1)\to 0.
\end{equation} 
Here note that since the sheaf
${^1}(\cE,\cO_C(a+1))^\vee\otimes \cO_C(a+1)$ appears as a subsheaf of the
torsion-free sheaf
$
\cF(-C)
$, it has to vanish.

Since there is no morphism from $\cF(-C)$ to $\cO_C(a+1)$ by the equality 
$(\cF,\cT)=0$ in \pref{lm:dimensions_of_ext_groups_in_TEF_sequence},
we see
$
( \cF(-C) , \cE )
\cong
( \cF(-C) , \cF(-C) )
=
k
$.
Hence the dual of the arguments before tells us that
the two morphisms
$
\cF(-C) \hookrightarrow \cE
$
in \pref{eq:restriction_of_E} and \pref{eq:F(-C)_exact_sequence} are 
isomorphic to each other. This implies that 
the sequence \pref{eq:FET_sequence} is isomorphic to 
the sequence \pref{eq:the_twist_sequence} as desired. \qed

%
%

\section{Exceptional sheaves sharing classes in $K_0$}\label{sc:Exceptional_sheaves_sharing_classes}
Recall that every exceptional object on a del Pezzo surface $X$ is
isomorphic to a shift of either an exceptional vector bundle
or a line bundle on a $ ( -1 ) $-curve.
If $\cE$ and $\cE'$ are exceptional vector bundles on $X$
satisfying the equality
\[
 \frac{c_1(\cE)}{\rank{\cE}} = \frac{c_1(\cE')}{\rank{\cE'}} ,
\]
then by \cite[Corollary 2.5]{MR966982}
they are isomorphic.
Therefore exceptional objects on
$ X $ are uniquely determined by their classes
in $ K_0 ( X ) $, up to even shifts
$ \bZ [ 2 ] $. 

Such uniqueness is no longer true for exceptional objects on $ \bF_2 $.
In Subsections \ref{subsc:Excinotvec} and \ref{subsc:Excivec},
starting with any exceptional sheaf $\cE$ on $\bF_2$,
we construct exceptional objects $\cE_i$ for $i\in \bZ$ sharing the classes in $K_0(\bF_2)$ with $\cE$.
Some of them are sheaves, and in \pref{subsc:sharing_sheaves}
we prove that thus constructed sheaves exhaust the set of exceptional sheaves sharing the class with
$\cE$.

\subsection{Construction of $\cE_i$ for exceptional sheaves with nontrivial torsion part}\label{subsc:Excinotvec}
Suppose that $\cE$ is an exceptional sheaf whose torsion part is nontrivial.
In this case, by \pref{th:main_conjecture;the_case_of_sheaves}, we have a short exact sequence
$$
0
\to
\cO_C(a_\cE)^{\oplus r_\cE}
\to
\cE
\to
\cF\cong T_{\cO_C(a_\cE)}\cE
\to
0
$$
for some integers $a_\cE$ and $r_\cE>0$.
Let us define integers $b_\cE$ and $s$ ($0<s\le \rank \cE$) by the isomorphism 
$$
\cF|_C\cong \cO_C(b_\cE)^{\oplus s}\oplus \cO_C(b_\cE+1)^{\oplus\rank \cE-s}
$$
(again use \cite[Remark 2.3.4]{MR1604186}).

Since $r_\cE={}^1 (\cF,\cO _C(a_\cE))$, we get the equality 
$$
r_\cE=(b_\cE-a_\cE)\cdot \rank \cE-s.
$$
For each integer $i$, set
\begin{align*}
\cF_i&:=
\cF \otimes \cO((b_\cE-a_\cE-1-i)C) \ \mathrm{and}\\
b_0&:=2a_\cE-b_\cE+2.
\end{align*}
Then we see
$
\cF= \cF_{b_\cE-a_\cE-1} 
$
and obtain the isomorphism 
\begin{equation*}\label{F|_C}
\cF_i|_C\cong \cO_C(b_0+2i)^{\oplus s}\oplus \cO_C(b_0+2i+1)^{\oplus\rank \cE-s}.
\end{equation*}
Finally we set 
$$
 \cE_i:= T_{\cO_C(b_0+i)}\cF_{i+1}.
$$
Then notice that 
we have an isomorphism
$$
\cE_i\cong  T'_{\cO_C(b_0+i-1)}\cF_i
$$ 
by the isomorphisms of the functors
\begin{equation*}\label{eq:lemma_4.15_IU05}
T_{\cO_C(a-1)}\circ T_{\cO_C(a)}\cong \otimes \cO (C)
\end{equation*}
for any $a\in \bZ$ (\cite[Lemma 4.15]{Ishii-Uehara_ADC}).
Since
$
a_\cE = b_0 + ( b_\cE-a_\cE-1)-1
$
and
$
\cF
=
\cF_{b_\cE-a_\cE-1}
$, we obtain the isomorphisms 
$$
\cE\cong  T'_{\cO_C(a_\cE)}\cF\cong \cE_{b_\cE-a_\cE-1}.
$$

%
%

\subsection{Relations among the objects $\cE_i$}\label{subsc:Exciproperties}

Let $\cE$ be an exceptional sheaf with nontrivial torsion part
 on $\bF_2$ as in Subsection \ref{subsc:Excinotvec}.
We study properties of $\cE_i$ constructed in
\pref{subsc:Excinotvec}.

The relationship among the objects
$\cF_i$
and
$\cE_i$
is summarized in the following diagram.

\begin{align*}
\xymatrixcolsep{3pc}
\xymatrix{
& 
\cF_{i+1}
\ar@/^{1pc}/[rr]^{\otimes \cO (C)}
\ar@{~>}[rd]^{T_{\cO_C(b_0+i)}} 
&&
\cF_i
\ar@/^{1pc}/[rr]^{\otimes \cO (C)}
\ar@{~>}[rd]^{T_{\cO_C(b_0+i-1)}} 
&&
\cF_{i-1}
\ar@{~>}[rd]^{T_{\cO_C(b_0+i-2)}}
&
\\
\cE_{i+1}
\ar@{~>}[ru]_{T_{\cO_C(b_0+i)}}
&&
\cE_{i}
\ar@{~>}[ru]_{T_{\cO_C(b_0+i-1)}} 
&&
\cE_{i-1}
\ar@{~>}[ru]_{T_{\cO_C(b_0+i-2)}}
&&
\cE_{i-2}
}
\end{align*}

\begin{claim}\label{cl:E_i_and_E_have_the_same_class}
We have
$[\cE_i]=[\cE]$ in $K_0(\bF_2)$ for any $i\in \bZ$. 
\end{claim}

\begin{proof}
As we have seen above, there exists an isomorphism
$\cE\cong \cE_i$ for some $i\in \bZ$. 
Since
$
T^2_{\cO_C(a)}
$
acts on
$
K_0(\bF_2)
$
trivially (see \pref{rm:spherical_twist_on_K_0}),
the result follows from the isomorphisms
$
T^2_{\cO_C(b_0+i)}\cE_{i+1}
\cong
\cE_i
$ for all $i\in \bZ$.
\end{proof}

\begin{claim}
\begin{enumerate}[(1)]
\item
For any integer  $i\ge -1$, $\cE_i$ is an exceptional sheaf.
\item
For any integer  $i<-1$, $\cE_i$ is an exceptional object of length $1$.
\item
$\cE_i$ is a vector bundle if and only if either $i=-1$, \emph{or}
$i=0$ and $s=\rank \cE$.
\item
There exists an isomorphism 
$$
\cE_{-1}|_C\cong \cO_C(b_0-1)^{\oplus\rank \cE-s}
\oplus
\cO_C(b_0)^{\oplus s}.
$$
\end{enumerate}
\end{claim}

\begin{proof}
We start with some preparatory computations for the twist functors.

Put
$
 r_i=(i+1)\rank \cE-s \in \bZ.
$ 
For each $i\ge 0$ we have $r_i\ge 0$, 
and $r_i=0$ occurs precisely when $i=0$ and $s=\rank \cE$.
The following calculation 
\begin{align}\label{eq:RHom_computation}
 & \RHom_{D ( \bF_2 ) } (\cF_i,\cO _C(b_0+i-1))
\notag\\
\cong &
 \Ext^1_{C} \left( \cO_C(b_0+2i)^{\oplus s} \oplus
 \cO_C(b_0+2i+1)^{\oplus \rank \cE-s}
 , \cO_C ( b_0 + i -1 ) \right) [ -1 ]\notag
\\
 \cong &  k^{\oplus r_i}[-1]
\end{align}
tells us that the defining exact triangles of the twist functor
$T_{\cO_C(b_0+i-1)}$ and its quasi-inverse $T'_{\cO_C(b_0+i-1)}$,
respectively, are equivalent to the following short exact sequences.
\begin{align}
0
\to
\cF_i \to  \cE_{i-1}= T_{\cO_C(b_0+i-1)}\cF_i
\to
\cO_C(b_0+i-1)^{\oplus r_i}
\to
0 \label{eq:twistF_i}
\\
0\to \cO_C(b_0+i-1)^{\oplus r_i}\to \cE_i\cong T'_{\cO_C(b_0+i-1)}\cF_i\to 
\cF_i\to 0 \label{eq:inversetwistF_i}
\end{align}

For each $i<0$, we have $r_i<0$.
In this case the calculation
\begin{align*}
 &
 \RHom_{ D ( \bF_2 ) } \left( \cF_i , \cO _C ( b_0 + i -1 ) \right)
\\
 \cong &
 \Hom_{ C } \left( \cO_C(b_0+2i)^{\oplus s}
 \oplus
 \cO_C(b_0+2i+1)^{\oplus \rank \cE-s},\cO_C(b_0+i-1) \right)
\\
 \cong & \   k^{\oplus -r_i}
\end{align*}
tells us that the defining triangles of the functors
$
 T_{\cO_C(b_0+i-1)}
$
and
$
 T '_{\cO_C(b_0+i-1)}
$
are as follows.
\begin{align}
\cO_C(b_0+i-1)^{\oplus -r_i}[-2]\to \cF_i\to 
 \cE_{i-1}= T_{\cO_C(b_0+i-1)}\cF_i \label{eq:twistF_i<0}\\
 \cE_{i}\cong T'_{\cO_C(b_0+i-1)}\cF_i\to \cF_i\to \cO_C(b_0+i-1)^{\oplus -r_i} \label{eq:inversetwistF_i<0}
\end{align}

Now we use all these results to obtain the conclusions.
The statement (1) follows from the exact sequence  \eqref{eq:twistF_i}.

To see (2),
note that \eqref{eq:twistF_i<0} implies
\[
\scH^p(\cE_{i-1})=
\begin{cases}
\cF_i & ( p = 0 )\\
\cO_C(b_0+i-1)^{\oplus -r_i} & ( p = 1 ) \\
0 & ( \mathrm{otherwise} ).
\end{cases}
\]

Next we check the assertion (3).
If
$
i \ge 1
$,
\eqref{eq:inversetwistF_i} implies that
$\cE_i$ has torsion.
If
$
 i = -1
$,
from \eqref{eq:inversetwistF_i<0}
we see that $\cE_{-1}$
has no torsion and hence is a vector bundle.
When
$
i = 0
$,
if we further assume
$
r_0=0
\iff
s=\rank \cE
$,
from \eqref{eq:inversetwistF_i} and \eqref{eq:twistF_i}
we obtain 
$
\cE_0\cong \cF_0\cong \cE_{-1}.
$
Hence
$
 \cE_0
$
is a vector bundle in this case.
Finally if $r_0>0$, then \eqref{eq:inversetwistF_i}
forces $\cE_0$ to have torsion.

To show the last assertion (4),
use the isomorphism
$
\scTor_1^ {\bF_2}(\cO_C,\cO_C(b_0-1))\cong \cO_C(b_0+1)
$
and the restriction of \eqref{eq:twistF_i} for $i=0$ 
to $C$.
\end{proof}

\begin{remark}\label{rem:E_iE_j}
By the above proof, we know that 
$\cE_i\cong \cE_j$ if and only if either $i=j$ \emph{or} 
$\{i,j\}=\{-1,0\}$ and $s=\rank \cE$.
\end{remark}

%
%

\subsection{Construction of $\cE_i$ for exceptional vector bundles}\label{subsc:Excivec}
In Subsection \ref{subsc:Excinotvec}, we constructed an 
exceptional vector bundle $\cE_{-1}$ which shares the class in 
$K_0(\bF_2)$ with 
a given exceptional sheaf $\cE$ with nontrivial torsion part. 
In this section, let us follow the procedure in an opposite direction;
starting with an exceptional vector bundle  
$\cE_{-1}$ on $\bF_2$, 
let us recover exceptional objects $\cE_i$ 
for each $i$. See the precise statement in \pref{rm:compatible}.

Suppose that an exceptional vector bundle $\cE'$ is given.
Set 
$$
\cE'_{-1}:=\cE'.
$$
Since $\cE'_{-1}|_C$ is rigid 
by \cite[Remark 2.3.4]{MR1604186}, there exist integers $b'_0$ and $s'$  ($0<s'\le \rank \cE'$) such that
$$
\cE'_{-1}|_C
\cong
\cO_C(b'_0-1)^{\oplus\rank \cE'-s'}\oplus\cO_C(b'_0)^{\oplus s'}.
$$
Next let us define 
$$
 \cF'_0:= T'_{\cO_C(b'_0-1)}\cE'_{-1}.
$$
Then we see 
\begin{align*}
 & \RHom_{D ( \bF_2 ) } (\cE'_{-1},\cO _C(b'_0-1))
\\
 \cong &
 \left( \cO_C( b'_0 - 1 )^{ \oplus \rank \cE' - s'}
 \oplus
 \cO_C ( b'_0 )^{ \oplus s'} ,
 \cO_C ( b'_0 - 1 ) \right)
\\
 \cong &
 k^{ \oplus \rank \cE' - s' },
\end{align*}
so as to obtain the following short exact sequence
\begin{equation}\label{eqn:F0E-1}
0\to \cF'_0=T'_{\cO_C(b'_0-1)}\cE'_{-1}\to \cE'_{-1}\to \cO_C(b'_0-1)^{\oplus\rank \cE'-s'}\to 0.
\end{equation}
This implies that $\cF'_0$ is also an exceptional vector bundle.
Successively we define 
$$
 \cF'_i := \cF'_0 ( - i C ) , \ \cE'_i := T_{ \cO_C ( b'_0 + i ) } \cF'_{ i + 1 }
$$
for any $i\in \bZ$.
Using the isomorphism
$
\scTor_1^ {\bF_2}(\cO_C,\cO_C(b_0-1))\cong \cO_C(b_0+1)
$
and restricting \eqref{eqn:F0E-1} to $C$, we obtain 
$$
\cF'_0|_C\cong \cO_C(b'_0)^{\oplus s'}\oplus \cO_C(b'_0+1)^{\oplus \rank \cE'-s'},
$$
and 
\begin{equation*}\label{F|_C}
\cF'_i|_C\cong \cO_C(b'_0+2i)^{\oplus s'}\oplus \cO_C(b'_0+2i+1)^{\oplus\rank \cE'-s'}.
\end{equation*}
Then a direct computation as in \pref{eq:RHom_computation}
 yields that $\cE'_i$ is a sheaf if and only if $i\ge -1$.

\begin{remark}\label{rm:compatible}
We can see, from the definitions, that the constructions of $\cE_i$
and $\cE'_i$  given in Subsections \ref{subsc:Excinotvec} and \ref{subsc:Excivec},
respectively, are inverses of each other
in the following sense.

Given an exceptional sheaf
$\cE$ with nontrivial torsion part,
construct the exceptional objects
$\cE_i$ as in \pref{subsc:Excinotvec}.
From the exceptional vector bundle
$\cE':=\cE_{-1}$,
we can also construct the exceptional objects
$\cE'_i$ as in \pref{subsc:Excivec}.
Then we see 
$
 \cE_i\cong \cE'_i
$ 
and
$
 \cF_i\cong \cF'_i
$
for each
$
 i \in \bZ
$.
In fact,
$
 \cE_{-1}
 \cong
 \cE'_{-1}
$
holds by definition.
Since the objects
$
 \cE_i , \cF_i
$
for
$
 i \in \bZ
$
and
$
 \cE'_i , \cF'_i
$
for
$
 i \in \bZ
$
satisfy the same recursive relations,
we obtain the conclusion;
note that the numbers
$
 b_0
$
and 
$
b'_0
$
in
Subsections \ref{subsc:Excinotvec} and \ref{subsc:Excivec},
respectively,
depend only on the restrictions to
$
 C
$
of
$
 \cE_{-1}
$
and
$
 \cE'_{-1}
$.
Since
these two sheaves are known to be isomorphic,
we obtain $b_0=b'_0$.

Similarly, given an exceptional vector bundle $\cE'$,
construct the exceptional objects
$\cE'_i$ as in \pref{subsc:Excivec}.
Choose an integer $i \ge 0$ such that
$\cE:=\cE'_i$ is an exceptional sheaf with nontrivial torsion part,
and construct the exceptional objects
$\cE_i$ as in \pref{subsc:Excinotvec}.
Then we see 
$
 \cE_i\cong \cE'_i
$ 
and
$
 \cF_i\cong \cF'_i
$
for each
$
 i \in \bZ
$.
To see this, it is again enough to check
$
 \cE_{-1}
 \cong
 \cE'_{-1}
$.
Note, by their constructions,
that they are exceptional vector bundles whose classes in
$
 K_0( \bF_2 )
$
are the same.
Then we can use \pref{lem:Gor_Kul} to see that
these two vector bundles should be isomorphic.
\end{remark}

\subsection{Exceptional sheaves sharing the same class}\label{subsc:sharing_sheaves}

We start with the following lemma.

\begin{lemma}\label{lem:Gor_Kul}
Let $\cE$ and $\cE'$ be exceptional vector bundles on $\bF _2$.
Suppose that 
the equality $\frac{c_1(\cE)}{\rank{\cE}}=\frac{c_1(\cE')}{\rank{\cE'}}$ 
holds in
$
\Pic{(\bF_2)}_{\bQ}
$.
Then $\cE\cong \cE'$.
\end{lemma}

\begin{proof}
The proof is analogous to that of 
\cite[Corollary 2.5]{MR966982} in the case of del Pezzo surfaces. The $(-K)$-stability of exceptional vector bundle on $\bF_2$ follows from 
\cite[Corollary 2.2.9]{MR1604186}.
\end{proof}

Now we give the main result of this section.


\begin{theorem}
Suppose that $\cE$ and $\cE'$ are exceptional sheaves on $\bF _2$.
Then the following conditions are equivalent.
\begin{enumerate}[(1)]
\item
$
 \frac{c_1(\cE)}{\rank{\cE}}
 =
 \frac{c_1(\cE')}{\rank{\cE'}}
 \in \Pic{(\bF_2)}_{\bQ}
$
.
\item
$
\cE_i \cong \cE'_i
$
for any
$
 i \in \bZ
$.

\item
$
 [\cE] = [\cE'] \in K_0(\bF_2)
$.

\item
$
 \cE'\cong \cE_i
$
for some
$
 i \ge -1
$.
\end{enumerate}
\end{theorem}

\begin{proof}
(4)
$
\Rightarrow
$
(3) easily follows from \pref{cl:E_i_and_E_have_the_same_class}.
(3)
$
\Rightarrow
$
(1)
is also obvious, since the Chern character map factors through the
$
K
$-group.

As we have seen in
Sections \ref{subsc:Excinotvec} and \ref{subsc:Excivec},
$ \cE' \cong \cE'_{i} $ always holds for some $ i $. This implies
(2)
$
\Rightarrow
$
(4).

Finally let us assume (1).
By
\pref{cl:E_i_and_E_have_the_same_class}
we obtain the equalities
$
[\cE]=[\cE_{-1}]
$
and
$
[\cE']=[\cE'_{-1}]
$.
Hence we see
$
\frac{c_1(\cE_{-1})}{\rank \cE_{-1}}
= \frac{c_1(\cE'_{-1})}{\rank \cE'_{-1}}
$. 
Since
$
\cE_{-1}
$
and
$
\cE'_{-1}
$
are vector bundles,
we obtain an isomorphism
$
\cE_{-1} \cong \cE'_{-1}
$
by \pref{lem:Gor_Kul}.
Thus we obtain the condition $(2)$. 
\end{proof}

We summarize the results of this section.

\begin{corollary}\label{cr:classification_of_equivalent_sheaves}
Let $\cE$ be an exceptional sheaf on $\bF_2$.
Then the set of isomorphism classes of exceptional sheaves
$\cE'$ with $[\cE]=[\cE'] \in K_0(\bF_2)$ is just 
$
\{
\cE_i\mid i\ge -1
\}.
$
In this set, $\cE_{-1}$ is the unique vector bundle up to isomorphism.
Furthermore, if $ i \not= j $,
$\cE_i\cong \cE_j$ occurs if and only if 
$
\{ i , j \} = \{ -1 , 0 \}
$
and
$
\cE_{-1}|_C
\cong
\cO_C (b)^{\oplus\rank \cE_{-1}}
$
for some
$b \in \bZ$. 
\end{corollary}

%
%

\section{Some results obtained via deformation to del Pezzo surfaces}

We explain a couple of results about exceptional objects on
$\bF_2$, which are obtained by using its deformation to
$\bF_0
\cong
\bP^1 \times \bP^1$.
Since
$\bF_0$
is a del Pezzo surface,
exceptional objects on it are well known
(see \cite{MR1286839} and \cite{MR2108443})).
The idea is to use that knowledge to understand
the exceptional objects on
$\bF_2$.

\subsection{Any exceptional object has, up to sign, the
same class as an exceptional bundle}

\begin{proposition}\label{pr:existence_of_sheaf_in_the_class}
Let
$
n
$
be an even non-negative integer.
Then for any exceptional object
$
\cE
$
on the Hirzebruch surface
$\bF_n$,
there exists an exceptional vector bundle
$\cF$
on
$\bF_n$
such that
\[
[\cE]=[\cF]\  \mathrm{or} \ 
-[\cF] \in K_0(\bF_n).
\] 
\end{proposition}
We need some preparations for the proof.
The assumption that
$
n
$
is even is used only in the proof of
\pref{cr:lifting_exceptional_objects_to_F_0}.

\begin{lemma}\label{lm:degeneration_of_hirzebruch_surfaces}
Let
$
n
$
be a non-negative integer and
$
m
=
1
$
or
$
0
$,
with the same parity as
$
n
$.
Then there exists a smooth projective morphism
\begin{equation}\label{eq:degeneration_of_hirzebruch_surfaces}
f \colon \scX \to \bA^1
\end{equation}
such that
\begin{itemize}
\item $\scX_0 \cong \bF_n$,
\item
$
f^{-1}(\bA^1\setminus \{0\}) \cong \bF_m \times (\bA^1\setminus \{0\})
 \ \mathrm{over} \ \bA^1\setminus \{0\}.
$
\end{itemize}
\end{lemma}

\begin{proof}
See \cite[Example 2.16]{MR815922}.
we can see that the construction works in any characteristics.
\end{proof}

\begin{lemma}\label{lm:huybrechts_thomas}
Let
\[
\scX_R \to \Spec R
\]
be a smooth projective (for simplicity) morphism over the spectrum of a
complete discrete valuation ring
$R$.
Let
$\scX_0$
be the central fiber and take an object
$\cE_0 \in D(\scX_0)$.

\begin{itemize}
\item
When
$
\Ext^2_{\scX_0}(\cE_0, \cE_0)=0
$,
$\cE_0$
extends to a bounded complex of coherent sheaves on
$\scX_R$.

\item
When
$\Ext^1_{\scX_0}(\cE_0, \cE_0)=0$,
the extension of
$\cE_0$
to
$\scX$,
if exists, is unique.
\end{itemize}
\end{lemma}
\begin{proof}
By \cite[Corollary 3.4]{MR2578562}, we see that
the vanishing of
$\Ext^2_{\scX_0}(\cE_0, \cE_0)$
implies that $\cE_0$ can be extended to a perfect complex
on the formal neighborhood of
$\scX_0$.
By the algebraization \cite[Theorem 8.4.2]{MR2222646},
we can extend the perfect complex to
$\scX$.
For the uniqueness, see again \cite[Corollary 3.4]{MR2578562}.
\end{proof}

Let
$
R
$
be the completion of the local ring
$\cO_{0, \bA^1}$
by the maximal ideal,
and
$
K
$
the field of fractions of
$
R
$.
Let
\begin{equation}\label{eq:degenaration_of_Hirzebruch_surface_after_completion}
\scX_R \to \Spec R
\end{equation}
be the base change of the morphism
\pref{eq:degeneration_of_hirzebruch_surfaces} by
the natural morphism $\Spec R \to \bA^1$.

\begin{corollary}\label{cr:lifting_exceptional_objects_to_F_0}
Assume that
$
n
$
is an even non-negative integer. Then any exceptional object
$\cE_0$
on
$\scX_0\cong \bF_n$
uniquely deforms to a family of exceptional objects
$\cE$
on
$\scX_R$
over
$R$.
Moreover, the restriction 
$\cE_K:=\cE|_{\scX_K}$
to the generic fiber is isomorphic to a shift of an exceptional vector bundle.
\end{corollary}
\begin{proof}
Although
$K$
is not algebraically closed,
by passing to the algebraic closure,
we can check that
any exceptional object on $\scX_K$
is isomorphic to a shift of a vector bundle
by applying \cite[Propositions 2.9 and 2.10]{MR1286839}.
Note that this is not the case for the Hirzebruch surface
$\bF_1$, since line bundles on the
$(-1)$-curve are also exceptional; this is why
we assumed that
$
n
$
is even. 

The rest is a direct consequence of
\pref{lm:huybrechts_thomas}.
\end{proof}

Next we recall a fact on the relative moduli space of semi-stable sheaves.
\begin{lemma}[{$=$\cite[Theorem 4.1]{MR2085175}}]\label{lm:Langer}
Let
$
\scY
\to 
B
$
be a projective morphism of schemes of finite type over
a universally Japanese ring, with a relatively ample line bundle
$
L
$
on
$
\scY
$.
Fix an integer-valued polynomial
$
P(t)
\in
\bQ[t]
$.
Then
there exists a projective scheme
$
\Mbar
\to
B
$
such that for any
$
b
\in
B
$, the fiber
$
\Mbar_b
$
is the moduli scheme of
$L_b$-semi-stable
sheaves on
$\scY_b$
with the Hilbert polynomial
$P(t)$.
\end{lemma}

\begin{lemma}\label{lm:existence_of_degeneration}
Fix a non-negative integer
$
n
$.
Let
\[
\scX_R \to \Spec R
\]
be the morphism
\pref{eq:degenaration_of_Hirzebruch_surface_after_completion}
and
$\cE_K$
an exceptional vector bundle on the generic fiber
$\scX_K$.
Then there exists an exceptional vector bundle
$\cF_0$
on the central fiber
$\scX_0$
such that
\begin{equation}\label{eq:classes_coincides}
[\cE_K]
=
[\cF_0]
\end{equation}
under the canonical isomorphism
$
K_0((\bF_n)_K)
\cong
K_0(\bF_m)
$.
\end{lemma}

\begin{proof}
Note first that there exists a canonical isomorphism
\[
\Pic((\bF_m)_K)\cong\Pic(\bF_n),
\]
so that we can identify the $\bR$-divisors
on
$
(\bF_n)_K
$
with those on
$
\bF_m
$.

Note second that
$\cE_K$
is
$(-K_{(\bF_m)_K})$-stable (see \cite[Theorem 5.2]{MR1286839}).
Take an ample divisor
$H \in \Pic (\bF_n)\otimes \bR \cong \Pic ((\bF_m)_K) \otimes \bR$
which is sufficiently close to
$-K_{(\bF_m)_K}$
so that
$\cE_K$
is
$H$-stable.
Moreover
by the local finiteness of the walls
(see \cite[Section 4C]{Huybrechts-Lehn}), we can choose
$H$ generically so that there exists no $H$-strictly semi-stable sheaf
$\scG \in \Coh(\bF_n)$
with
$
[\scG]=[\cE_K]
$;
here the equality should be considered under the canonical isomorphism
$K_0(\bF_n)\cong K_0((\bF_m)_K)$.

Consider the relative moduli space of
$H$-semi-stable sheaves
\[
\Mbar \to \Spec R
\]
obtained by applying \pref{lm:Langer} to \pref{eq:degenaration_of_Hirzebruch_surface_after_completion}
for
$\scY=\scX_R$
and
$B=\Spec R$.
Since
$R$
is a DVR,
we can find a section
$s\colon \Spec R \to \Mbar$
such that
$
s(\Spec K)
$
represents the sheaf
$
\cE_K
$.
By the genericity of the polarization
$H$,
it holds that the section
$s$
factors through the stable locus
$
M
\subset
\Mbar
$.
Hence, pulling back by
$
s
\times
\id_{\scX_R}
$ the quasi-universal family on
$
M
$,
we obtain a flat family of coherent sheaves
$\cF'$
on
$
\scX_R
$
such that
$\cF'_K
\cong
\cE_K^{\oplus r}
$
for some
$r>0$
(see \cite[Section 4.6]{Huybrechts-Lehn} for the notion of
quasi-universal families).
If we denote by
$\cF_0$
the stable sheaf represented by the point
$s(0)
\in
\Mbar_{0}
$,
then by the definition of the quasi-universal family,
we see
$\cF'_0
\cong
\cF_0^{\oplus r}$.

Therefore we see that
$\cF_0$
is torsion free,
${}^0(\cF_0,\cF_0)=1$,
and
${}^2(\cF_0,\cF_0)={}^0(\cF_0,\cF_0\otimes \omega_{\bF_n})=0$.
Moreover, note that
\[
\begin{split}
\chi(\cF_0,\cF_0) =\frac{1}{r^2}\chi(\cF'_0,\cF'_0)
=\frac{1}{r^2}\chi(\cF'_K,\cF'_K)\\
=\chi(\cE_K,\cE_K)=1
\end{split}
\]
due to the deformation invariance of
the Euler pairing. This implies
\[
{}^1(\cF_0, \cF_0)
=
-\chi(\cF_0,\cF_0)+{}^0(\cF_0,\cF_0)+{}^2(\cF_0,\cF_0)
=
-1+1+0
=0.
\]
Therefore
$\cF_0$
is an exceptional torsion free sheaf. By standard arguments (see \cite[Corollary 2.3]{MR1286839}),
$\cF_0$ has to be a vector bundle.
Finally the sequence of equalities
\[
[\cE_K]
=
\frac{1}{r}
[\cF'_K]
=
\frac{1}{r}
[\cF'_0]
=
[\cF_0]
\]
imply
\pref{eq:classes_coincides},
since
$
K_0((\bF_n)_K)
\cong
K_0(\bF_m)
$
is torsion free.
\end{proof}

Finally,
\pref{pr:existence_of_sheaf_in_the_class}
directly follows from
\pref{cr:lifting_exceptional_objects_to_F_0}
and
\pref{lm:existence_of_degeneration}.

\begin{remark}\label{rm:action_of_double_spherical_twists}
It is interesting to describe the set of exceptional objects
sharing the class in
$K_0(\bF_n)$.
When
$n=2$,
because of
\pref{rm:spherical_twist_on_K_0},
the nontrivial group generated by `double spherical twists'
\begin{equation*}\label{eq:group_of_double_spherical_twists}
D:=\langle T_{\alpha}^2 \ | \ \alpha\colon\mathrm{spherical} \rangle < \Auteq{D(\bF_2)}
\end{equation*}
acts trivially on
$K_0(\bF_2)$.

Note that the action of
$D$
is not free: for example
$
T_{\cO_C(-1)}\cO=\cO,
$
since
$
\RHom(\cO_C(-1), \cO)=0.
$
The authors are not sure if the action is transitive
on each of such sets or not.
\end{remark}

%
%

\subsection{Numerical transitivity}
In this section we consider
$\bF_n$
for any
$n \ge 0$.
Let
$B_4$
be the braid group with four strings.
It is well known (see \cite[Proposition 2.1]{MR1230966}) that the group
$G_4:=\bZ^4 \rtimes B_4$ acts on the set of
full exceptional collections of a triangulated category
$\scrT$ with
$K_0(\scrT)\cong\bZ^4$, via right and left mutations and the
shifts.
The action descends to
$
K_0(\scrT)
$
(see \cite[Section 1.5]{MR2108443}).
Below is a special case of
\cite[p.216, Corollary]{MR1073089}.

\begin{proposition}
Let
\begin{equation}\label{eq:exceptional_collection_E_1_E_2_E_3_E_4}
\cE_1, \cE_2, \cE_3, \cE_4
\end{equation}
be a full exceptional collection of
$D(\bF_n)$.
Then for any full exceptional collection of line bundles
\begin{equation}\label{eq:standard_collection}
S_1 , S_2 , S_3 , S_4,
\end{equation}
there exists an element
$g \in G_4$
such that the classes of the members of the mutated collection
\[
g(\cE_1, \cE_2, \cE_3, \cE_4)
\]
in
$K_0(\bF_n)$
coincide with those of the chosen collection
\pref{eq:standard_collection}.
\end{proposition}

\begin{proof}
Because of the existence of the deformation
\pref{eq:degenaration_of_Hirzebruch_surface_after_completion},
there exists a canonical isomorphism
\begin{equation*}
\varphi\colon K_0(\bF_n)\xrightarrow[]{\cong} K_0((\bF_m)_K) 
\end{equation*}
which preserves the Euler pairings on both sides.

Consider the deformation of the collections
\pref{eq:exceptional_collection_E_1_E_2_E_3_E_4}
and
\pref{eq:standard_collection}
to
$(\bF_m)_K$, which automatically are full exceptional collections on
$(\bF_m)_K$ by the upper semi-continuity of the cohomology and
\cite[Theorem 6.11]{MR1286839}.
Denote them by
\begin{equation}\label{eq:deformed_collection}
(\cE_1)_K,\dots,(\cE_4)_K
\end{equation}
and
\begin{equation}\label{eq:deformed_standard_collection}
(S_{1})_{K},\dots,(S_{4})_{K}
\end{equation}
respectively.
By the transitivity
\cite[Theorem 4.6.1]{MR2108443},
we can find an element
$g \in G_4$
which sends
\pref{eq:deformed_collection}
to
\pref{eq:deformed_standard_collection}.

On the other hand, by definition, the actions of
$G_4$
on the level of
$K_0$ are compatible with the isometry
$\varphi$.
In particular the classes in
$K_0(\bF_2)$
of the collection
\[
g(\cE_1,\dots,\cE_4)
\]
should be the same as those of the collection
\pref{eq:standard_collection}. Thus we conclude the proof.
\end{proof}

\begin{remark}
Let
$\scC'$
be the set of full exceptional collections of
$D(\bF_0)$, and
$\scC:=\scC'/\bZ^4$.
The action of
$B_4$ on
$\scC$
is known to be transitive, but not free.
To see this, consider the full exceptional collection
$\cO, \cO(0,1), \cO(1,0), \cO(1,1)$.
If we denote by
$\sigma_{23}\in B_4$
the flip of the 2nd and the 3rd threads, then
$\sigma_{23}^2$ fixes the collection; it is simply because
there is no derived morphism between
$\cO(0,1)$ and
$\cO(1,0)$.
Note that, on the other hand, $\sigma_{23}^2$ acts nontrivially
on the degeneration of the collection to
$\bF_2$.

It is intriguing to understand the stabilizer group of this action,
and to see if the action of that stabilizer subgroup
on the set of full exceptional collections of
$D(\bF_2)$
is realized by autoequivalences
\pref{eq:group_of_double_spherical_twists}
of
$D(\bF_2)$.
\end{remark}

\bibliographystyle{amsalpha}
\bibliography{mainbibs}

\noindent
Shinnosuke Okawa

Department of Mathematics,
Graduate School of Science,
Osaka University,
1-1 Machikaneyama,
Toyonaka,
Osaka,
560-0043,
Japan.

{\em e-mail address}\ : \  okawa@math.sci.osaka-u.ac.jp
\ \vspace{0mm} \\

\noindent
Hokuto Uehara

Department of Mathematics and Information Sciences,
Tokyo Metropolitan University,
1-1 Minamiohsawa,
Hachioji-shi,
Tokyo,
192-0397,
Japan.

{\em e-mail address}\ : \  hokuto@tmu.ac.jp
\ \vspace{0mm} \\

\end{document}